\documentclass[12pt]{article}
\usepackage{amstext}
\usepackage{amsthm}
\usepackage{amssymb}
\usepackage{ucs}
\usepackage{comment}
\usepackage{amsthm}
\usepackage{amsmath}
\usepackage{latexsym}
\usepackage[cp1251]{inputenc}
\usepackage[english]{babel}
\usepackage{graphicx}
\usepackage{wrapfig}
\usepackage{caption}
\usepackage{subcaption}
\usepackage{txfonts}
\usepackage{lmodern}
\usepackage{mathrsfs}
\usepackage{algorithm}
\usepackage{multicol}
\usepackage{enumerate}
\usepackage{titlesec}
\usepackage{paralist}
\usepackage{mathtools}
\usepackage{tikz}
\usetikzlibrary{shapes,snakes}
\titleformat{\paragraph}
{\normalfont\normalsize\bfseries}{\theparagraph}{1em}{}
\titlespacing*{\paragraph}
{0pt}{3.25ex plus 1ex minus .2ex}{1.5ex plus .2ex}

\newcommand*{\myproofname}{Proof}

\usepackage{wasysym}

\usepackage{fullpage}

\title{Packing chromatic number of subcubic graphs}
\date{\today}
\author{
J\' ozsef Balogh~\thanks{Department of Mathematics, University of Illinois at Urbana--Champaign, IL, USA, jobal@illinois.edu. 
Research of this author is partially supported by NSF Grant DMS-1500121, Arnold O. Beckman
Research Award (UIUC Campus Research Board 15006) and by the Langan Scholar Fund
(UIUC).}
\and Alexandr Kostochka \thanks{Department of Mathematics, University of Illinois at Urbana--Champaign, IL, USA and
Sobolev Institute of Mathematics, Novosibirsk 630090, Russia, kostochk@math.uiuc.edu. Research of this author is supported in part by NSF grant
 DMS-1266016 and by grants 15-01-05867  and 16-01-00499 of the Russian Foundation for Basic Research.}
 \and Xujun Liu\thanks{Department of Mathematics, University of Illinois at Urbana--Champaign, IL, USA,  xliu150@illinois.edu.}
 }

\makeatletter
\newcommand{\neutralize}[1]{\expandafter\let\csname c@#1\endcsname\count@}
\makeatother

\newtheorem{theo}{Theorem}

\newtheorem{lemma}[theo]{Lemma}

\newtheorem{corl}[theo]{Corollary}

\theoremstyle{definition}
\newtheorem{defn}[theo]{Definition}

\theoremstyle{remark}

	\def\quotient#1#2{%
		\raise1ex\hbox{$#1$}\Big/\lower1ex\hbox{$#2$}%
	}

	\renewcommand{\epsilon}{\varepsilon}

\begin{document}
	\maketitle
	
\begin{abstract}
A {\em packing $k$-coloring} of a graph $G$ is a partition of $V(G)$ into sets $V_1,\ldots,V_k$ such that for each $1\leq i\leq k$
the distance between any two distinct $x,y\in V_i$ is at least $i+1$. The {\em packing chromatic number}, $\chi_p(G)$, of a graph $G$ is the minimum $k$ such that $G$ has a packing $k$-coloring. Sloper showed that there are   $4$-regular graphs with arbitrarily large
packing chromatic number. The question whether the packing chromatic number of subcubic graphs is bounded appears in several papers.
We answer this question in the negative. Moreover, we show that for every fixed $k$ and $g\geq 2k+2$,  almost
every $n$-vertex cubic graph of girth at least $g$ has the packing chromatic number greater than $k$.
\end{abstract}

	\section{Introduction}	
% Graphs in this paper  are assumed to be simple, i.e., they cannot have parallel edges or loops; multigraphs may have multiple edges and loops. 
For a positive integer $i$,  a set $S$ of vertices in a graph $G$ is {\em $\; i$-independent } if the distance in $G$ between any two
distinct vertices of $S$ is at least $i+1$. In particular, a $1$-independent set is simply an independent set.

A {\em packing $k$-coloring} of a graph $G$ is a partition of $V(G)$ into sets $V_1,\ldots,V_k$ such that for each $1\leq i\leq k$,
the set $V_i$ is $i$-independent. The {\em packing chromatic number}, $\chi_p(G)$, of a graph $G$, is the minimum $k$ such that $G$ has a packing $k$-coloring.
%the distance between any two distinct $x,y\in V_i$ is at least $i+1$. 
 The of  notion packing $k$-coloring was introduced in 2008 by
Goddard, S.M. Hedetniemi, S.T.Hedetniemi,  Harris and  Rall~\cite{GHHHR}
(under the name {\em broadcast coloring}) motivated by frequency assignment problems in broadcast networks.
The concept has attracted a considerable attention recently: there are more than 25 papers on the topic 
(see e.g.~\cite{ANT1,BKR1,BKR2,BKRW1,BKRW2,CJ1, FG1,FKL1,G1,GT1,S1} and references in them). In particular, 
Fiala and Golovach~\cite{FG1} proved
that finding the  packing chromatic number of a graph 
is NP-hard even in the class of trees. Sloper~\cite{S1} showed that there are graphs with maximum degree $4$ and arbitrarily large
packing chromatic number. 

The question whether the packing chromatic number of all {\em subcubic} graphs (i.e., the graphs with maximum degree at most $3$) 
is bounded by a constant was not resolved. For example, Bre\v sar, Klav\v zar, Rall, and Wash~\cite{BKRW2} write:
{\em `One of the intriguing problems related to the packing chromatic number is whether it is bounded by a constant in the class of all cubic graphs'.} It was proved in~\cite{BKRW2,LBS1,S1} that it is indeed bounded in some subclasses of subcubic graphs. On the other hand,
 Gastineau and Togni~\cite{GT1} constructed a cubic graph $G$ with $\chi_p(G)=13$, and asked whether
there are cubic graphs with a larger packing chromatic number. Bre\v sar, Klav\v zar, Rall, and Wash~\cite{BKRW1} answered this question in affirmative by
constructing a cubic graph $G'$ with $\chi_p(G')=14$. The main result of this paper answers the question in full:
Indeed there are cubic graphs with arbitrarily large packing chromatic number. Moreover, we prove that `many' cubic graphs have
`high' packing chromatic number:

\begin{theo}\label{e1}
For each fixed integer $k\geq 12$ and $g\geq 2k+2$, %there is an $n_0(k,g)$ such that for each $n \geq n_0(k,g)$, 
 almost
every $n$-vertex cubic graph of girth at least $g$ has the packing chromatic number greater than $k$.
\end{theo}
%In particular, there are cubic graphs with arbitrarily large girth and arbitrarily large packing chromatic number.

The theorem will be proved in the language of the so-called {\em Configuration model}, $\mathcal{F}_3(n)$. 
We will discuss this concept and some important facts on it in the next section. In Section~3 
 we give upper bounds on the sizes $c_i$ of maximum $i$-independent sets in almost all 
cubic $n$-vertex graphs of large girth. The original plan was to show that for a fixed $k$ and large $n$, the sum $c_1+\ldots+c_k$ is less than $n$. 
But we were not able to prove it (and maybe this is not true). In Section 4, we give an upper bound on the size of the union of
an $1$-independent, a $2$-independent, and a $4$-independent sets which is less than $c_1+c_2+c_4$. This allows us  to 
 prove Theorem~\ref{e1} in the  last section.

\section{Preliminaries}\label{pr}
\subsection{Notation}
We mostly use standard notation.
 %The %complete $n$-vertex graph is denoted by $K_n$, and the 
% $n$-vertex cycle is denoted by $C_n$.
  If $G$ is a (multi)graph and $v,u \in V(G)$, then $E_G(v, u)$ denotes the set of all edges in $G$ connecting  $v$ and $u$,
 $e_G(v, u) \coloneqq |E_G(v, u)|$, and $\deg_G(v)\coloneqq\sum_{u\in V(G)\setminus \{v\}} e_G(v,u)$. 
For $A \subseteq V(G)$, $G[A]$ denotes the sub(multi)graph of $G$ induced by $A$.
%, $G[A,B]$ denotes the maximal bipartite sub(multi)graph of $G$ with parts $A$ and $B$. 
%If $G_1$, \ldots, $G_k$ are graphs, then $G_1 + \ldots + G_k$ denotes the graph with vertex set $V(G_1)\cup \ldots \cup V(G_k)$ and edge set $E(G_1) \cup \ldots \cup E(G_k)$. 
 The independence number of $G$ is denoted by $\alpha(G)$. For $k \in \mathbb{Z}_{> 0}$, $[k]$ denotes the set $\{1, \ldots, k\}$.

%Recall that a (proper) $k$-{\em coloring} of a graph $G$ is a mapping $f\,:\,V(G)\to [k]$ such that $f(v)\neq f(u)$ whenever $vu\in E(G)$.
%The smallest $k$ such that $G$ has a $k$-coloring is called the {\em chromatic number of} $G$ and is denoted by $\chi(G)$.

\subsection{The Configuration Model}
The configuration model is due in different versions to Bender and Canfield~\cite{BC1} and Bollob\'as~\cite{B2}. 
Our work is based on the version of Bollob\'as. Let $V$ be the vertex set of the graph,
 we are going to associate a 3-element set to each vertex in $V$. Let $n$ be an even positive integer.
 Let $V_n=[n]$ and consider the Cartesian product $W_n=V_n \times [3]$. 
A {\em configuration/pairing} (of order $n$ and degree $3$) is a partition of $W_n$ into $3n/2$ pairs, i.e.,~a perfect matching of
 elements in $W_n$. There are $$\frac{{3n \choose 2}\cdot {3n-2 \choose 2}\cdot\ldots\cdot {2 \choose 2}}{(3n/2)!}=(3n-1)!!$$ 
such matchings.
%different ways of doing that. 
Let $\mathcal{F}_3(n)$  denote the collection of all $(3n-1)!!$ possible pairings on $W_n$. 
 We project each pairing $F \in \mathcal{F}_3(n)$ to a multigraph $\pi(F)$ on the vertex set $V_n$ by ignoring the second coordinate. 
Then $\pi(F)$ is a $3$-regular multigraph (which may or may not contain loops and multi-edges). Let $\pi(\mathcal{F}_3(n))=\{\pi(F)\,:\, F\in\mathcal{F}_3(n)\}$ 
be the set of $3$-regular multigraphs on $V_n$. By definition,
\begin{equation}\label{0311}
\mbox{\em each simple graph $G\in \pi(\mathcal{F}_3(n))$ corresponds to  $(3!)^n$ distinct pairings in $\mathcal{F}_3(n)$.} 
\end{equation}
We will call the elements of $V_n$ - {\em vertices}, and of $W_n$ - {\em points}.
%\newpage

\begin{defn}
Let $\mathcal{G}_g(n)$ be the set of all cubic graphs with vertex set $V_n=[n]$ and girth at least $g$ and 
$\mathcal{G}'_g(n)=\{F\in \mathcal{F}_3(n)\,:\, \pi(F)\in \mathcal{G}_g(n)\}$.
\end{defn}

We will use the following result: % proved by Wormald~\cite{W1} and Bollob\'as~\cite{B2}.
\begin{theo}[Wormald~\cite{W1}, Bollob\'as~\cite{B2}]\label{BW}
For each fixed $g \geq 3$, 
%the portion of pairings of order $n$ and degree $3$ which yield simple graphs of girth at least $g$ is asymptotically 
\begin{equation}\label{03112}
\lim\limits_{n \to \infty}  \frac{|\mathcal{G}'_g(n)|}{|\mathcal{F}_3(n)|} = \exp \left \{ -\sum\limits_{k=1}^{g-1}\frac{2^{k-1}}{k} \right \}.
\end{equation}

\end{theo}

%\begin{remk}
{\bf Remark.}
When we say that {\em a pairing $F$ has a multigraph property $\mathcal{A}$}, we  mean that $\pi(F)$
has  property $\mathcal{A}$. 
%For example, an $i$-{\em independent set of} $F$ is a set of the form $I\times[3]$ where
%$I$ is an $i$-independent set in $\pi(F)$.
%\end{remk}

Since dealing with pairings is simpler than working with labeled simple regular graphs, we need the following %corollary, which is a 
well-known consequence of Theorem~\ref{BW}.

\begin{corl}[\cite{MK1}(Corollary 1.1),~\cite{JLR}(Theorem 9.5)]\label{MK2}
For fixed $g \geq 3$, any property that holds for $\pi(F)$ for almost all pairings  $F\in \mathcal{F}_3(n)$
%(as $n$ tends to infinity) 
 also holds for almost all graphs in $\mathcal{G}_g(n)$. %(as $n$ tends to infinity).
\end{corl} 
\begin{proof}
Suppose property $\mathcal{A}$ holds for $\pi(F)$ for almost all  $F\in\mathcal{F}_3(n)$.
% of degree $3$ on $W_n=V_n\times [3]$ % as $n$ tends to infinity. 
Let  $\mathcal{H}(n)$ denote the set of graphs in $\mathcal{G}_g(n)$ that do {\em not} have property $\mathcal{A}$
and $\mathcal{H}'(n)=\{F\in \mathcal{F}_3(n)\,:\, \pi(F)\in \mathcal{H}(n)\}$.
Let  $\mathcal{B}(n)$ denote the set of pairings   $F\in \mathcal{F}_3(n)$ such that $\pi(F)$  does { not} have property $\mathcal{A}$.
Then $\mathcal{H}'(n)\subseteq \mathcal{B}(n)$.
%We denote $A$ the set of all cubic graphs that have property $\mathcal{A}$ and $B$ the set of all pairings of degree $3$ that have property $\mathcal{A}$.  \\
Hence by the choice of $\mathcal{A}$,
\begin{equation}\label{equ1}
 \frac{|\mathcal{H}'(n)|}{|\mathcal{F}_3(n)|}\leq  \frac{|\mathcal{B}(n)|}{|\mathcal{F}_3(n)|}\to 0\qquad \mbox{as}\; {n\to\infty}.
\end{equation}

By~\eqref{0311}, we have
$$\frac{|\mathcal{H}(n)|}{|\mathcal{G}_g(n)|}=\frac{|\mathcal{H}(n)|}{|\mathcal{H}'(n)|}\cdot\frac{|\mathcal{H}'(n)|}{|\mathcal{G}'_g(n)|}\cdot 
\frac{|\mathcal{G}'_g(n)|}{|\mathcal{G}_g(n)|}=\frac{1}{(3!)^n}\cdot \frac{|\mathcal{H}'(n)|}{|\mathcal{G}'_g(n)|}\cdot {(3!)^n}=
\frac{|\mathcal{H}'(n)|}{|\mathcal{G}'_g(n)|}.$$
Furthermore,
\begin{equation}\label{equ3}
\frac{|\mathcal{H}'(n)|}{|\mathcal{G}'_g(n)|}=\frac{|\mathcal{H}'(n)|}{|\mathcal{F}_3(n)|}\cdot\frac{|\mathcal{F}_3(n)|}{|\mathcal{G}'_g(n)|}.
\end{equation}
By~\eqref{equ1} and Theorem~\ref{BW}, the  right-hand side of~\eqref{equ3}
 tends to $0$ as $n$ tends to infinity.
\end{proof}

%\begin{defn}
%Let $G$ be a graph. For a positive integer $i$, let $c_i(G)$ denote the maximum size of an $i$-independent set in $G$.  Sometimes we write $c_i$ instead of $c_i(G)$, if the graph $G$ is clear from the context.
%\end{defn}

%\newpage
%\section{Main result}
%\begin{theo}
%For each fixed $k$, there is an $n_0(k)$ such that for each $n\geq n_0(k)$, almost every $n$-vertex cubic
%graph $G$ has $\chi_p(G) > k$.
%\end{theo}	
%To prove our main result, we need the following lemmas and theorems.\\

\section{Bounds for $c_1,c_2,\ldots $}

We will use the following theorem of McKay~\cite{MK1}.

\begin{theo}[McKay~\cite{MK1}]\label{MK1}
 For every $\epsilon > 0$, there exists an $N>0$ such that for each $n>N$,  
 $${|\{ F |  F \in \mathcal{F}_3(n)\,:\;  c_1(\pi(F))>0.45537n\}|} \ <\epsilon \cdot {(3n-1)!!}.$$
\end{theo}

\begin{defn}
A {\em $3$-regular tree} is a tree such that each vertex has degree $3$ or $1$. 
%A {\em rooted balanced $3$-regular tree} is a $3$-regular rooted tree such that each pair of leaves  has the same distance to the root. 
A {\em $(3,k,a)$-tree} is  a rooted $3$-regular tree 
$T$ with root $a$ of degree $3$ such that the distance in $T$ from  each of the leaves  to $a$ is $k$.
\end{defn}

\begin{defn} For a positive integer $s$ and a vertex $a$ in a graph $G$, the {\em ball $B_{G}(a,s)$ in $G$ of radius $s$ with center $a$}
is $\{v\in V(G)\,:\; d_G(v,a)\leq s\}$, where $d_G(v,a)$ denotes the distance in $G$ from $v$ to $a$.
\end{defn}

%For a positive integer $g$, recall that $\mathcal{G}_g(n)$ denote the set of all $n$-vertex (labeled) cubic graphs with girth at least $g$.
We first prove  simple bounds on $c_{2k}(G)$ and $c_{2k+1}(G)$ when $G\in \mathcal{G}_{2k+2}(n)$.

\begin{lemma}\label{trivial bounds}
Let $j $ be a fixed positive integer and $g\geq 2j +2$.  Then\\

(i)  For every $ G \in \mathcal{G}_g(n)$, $c_{2j }(G) \leq \frac{n}{3 \cdot 2^{j }-2}$.\\

(ii) For every $\epsilon > 0$, there exists an $N>0$ such that for each $n>N$, 
$${\left|\left\{  G \in \mathcal{G}_g(n)\,:\; c_{2j +1}(G)> \frac{0.45537n}{2^{j +1}-1}\right\}\right|}\ <\epsilon \cdot {|\mathcal{G}_g(n)|} .$$
\end{lemma}

\begin{proof}
(i) Let $C_{2j }$ be a $2j $-independent set in $G$ with $|C_{2j }|=c_{2j }(G)$. Since the distance between
any distinct $a,b\in C_{2j }$ is at least $2j +1$, 
  the balls $B_{G}(a,j )$ for all distinct  $a\in C_{2j }$ are disjoint. Moreover, 
since $g \geq 2j +2$,  each ball $B_{G}(a,j )$ induces
 a $(3,j ,a)$-tree $T_a$, and hence has
 $$1+3+3\cdot2+3\cdot2^2+\ldots +3\cdot2^{j -1}=3\cdot2^j -2$$  vertices.  This proves (i).
 %A pair of trees $T_a$ and $T_b$, for $a,b \in C_{2j }$, have no vertex in common because $d(a,b) \geq 2j +1$.\\

(ii)  Let $C_{2j +1}$ be a $(2j +1)$-independent set in $G$ with $|C_{2j +1}|=c_{2j +1}(G)$. 
As in the proof of (i), 
  the balls $B_{G}(a,j )$ for  distinct  $a\in C_{2j }$ 
   are disjoint, and each $B_{G}(a,j )$ induces
 a $(3,j ,a)$-tree $T_a$. But in this case, in addition, the balls with centers in distinct vertices of
   $C_{2j +1}$ are at distance at least $2$ from each other. Let $S_i$ be the set of vertices in $T_a$ at distance $i$ from $a$.
Then $|S_0|=1$, and  for each $1\leq i\leq j $, $|S_i|=3\cdot 2^{i-1}$. If follows that the set 
$I_a=\bigcup_{i=0}^{\lfloor j /2\rfloor}S_{j -2i}$ is independent, and 
$$|I_a| = \sum_{i=0}^{\lfloor j /2\rfloor} |S_{j -2i}|=2^{j +1}-1.$$
Therefore $I:=\bigcup_{a\in C_{2j +1}}I_a$ is an independent set in $G$ and $|I|=(2^{j +1}-1) c_{2j +1}(G)$.
This together with Theorem~\ref{MK1} and Corollary~\ref{MK2} implies (ii).
\end{proof}

%picture

%\begin{comment}

\usetikzlibrary{shadows}

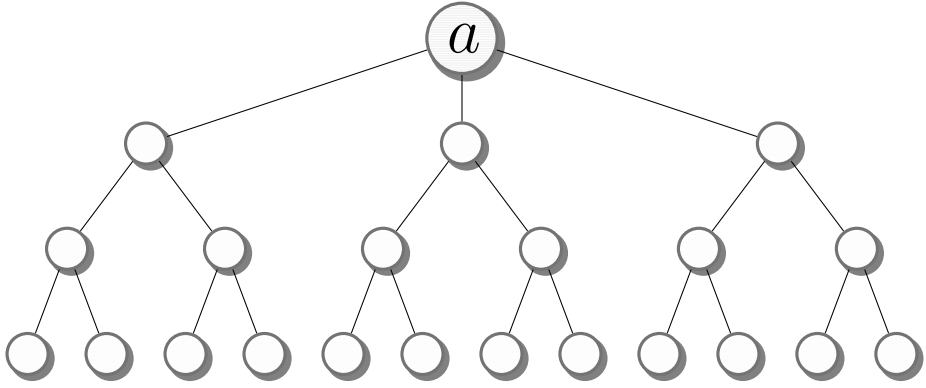
\begin{figure}[ht]\label{f1}

\begin{center}

%\begin{comment}

\begin{tikzpicture}[scale=0.7, rotate=270]
\tikzset{button/.style={
% First preaction: Fuzzy shadow
preaction={fill=black,path fading=circle with fuzzy edge 20 percent,
opacity=.5,transform canvas={xshift=1mm,yshift=-1mm}},
% Second preaction: Background pattern
preaction={pattern=#1,
path fading=circle with fuzzy edge 15 percent},
% Third preaction: Make background shiny
preaction={top color=white,
bottom color=black!50,
shading angle=45,
path fading=circle with fuzzy edge 15 percent,
opacity=0.2},
% Fourth preaction: Make edge especially shiny
preaction={path fading=fuzzy ring 15 percent,
top color=black!5,
bottom color=black!0,
shading angle=45},
inner sep=2ex
},
button/.default=horizontal lines light blue,
circle}

\node [ font={\huge\bfseries}, shape=circle, minimum size=0.5cm, circular drop shadow, text=black, very thick, draw=black!55, top color=white,bottom color=black!0, text width=0.4cm, align=center] (v1) at (0,0) {$a$};
\node [ font={\huge\bfseries}, shape=circle, minimum size=0.5cm, circular drop shadow, text=black, very thick, draw=black!55, top color=white,bottom color=black!0, text width=0.2cm, align=center] (v2) at (2,6) {};
\node [ font={\huge\bfseries}, shape=circle, minimum size=0.5cm, circular drop shadow, text=black, very thick, draw=black!55, top color=white,bottom color=black!0, text width=0.2cm, align=center] (v3) at (2,0) {};
\node [ font={\huge\bfseries}, shape=circle, minimum size=0.5cm, circular drop shadow, text=black, very thick, draw=black!55, top color=white,bottom color=black!0, text width=0.2cm, align=center] (v4) at (2,-6) {};
\draw (v1) edge (v2);
\draw (v1) edge (v3);
\draw (v1) edge (v4);
\node [ font={\huge\bfseries}, shape=circle, minimum size=0.5cm, circular drop shadow, text=black, very thick, draw=black!55, top color=white,bottom color=black!0, text width=0.2cm, align=center] (v5) at (4,7.5) {};
\node [ font={\huge\bfseries}, shape=circle, minimum size=0.5cm, circular drop shadow, text=black, very thick, draw=black!55, top color=white,bottom color=black!0, text width=0.2cm, align=center] (v6) at (4,4.5) {};
\node [ font={\huge\bfseries}, shape=circle, minimum size=0.5cm, circular drop shadow, text=black, very thick, draw=black!55, top color=white,bottom color=black!0, text width=0.2cm, align=center] (v10) at (4,1.5) {};
\node [ font={\huge\bfseries}, shape=circle, minimum size=0.5cm, circular drop shadow, text=black, very thick, draw=black!55, top color=white,bottom color=black!0, text width=0.2cm, align=center] (v7) at (4,-1.5) {};
\node [ font={\huge\bfseries}, shape=circle, minimum size=0.5cm, circular drop shadow, text=black, very thick, draw=black!55, top color=white,bottom color=black!0, text width=0.2cm, align=center] (v8) at (4,-4.5) {};
\node [ font={\huge\bfseries}, shape=circle, minimum size=0.5cm, circular drop shadow, text=black, very thick, draw=black!55, top color=white,bottom color=black!0, text width=0.2cm, align=center] (v9) at (4,-7.5) {};
\draw (v2) edge (v5);
\draw (v2) edge (v6);
\draw (v3) edge (v7);
\draw (v4) edge (v8);
\draw (v9) edge (v4);
\draw (v3) edge (v10);
\node [ font={\huge\bfseries}, shape=circle, minimum size=0.5cm, circular drop shadow, text=black, very thick, draw=black!55, top color=white,bottom color=black!0, text width=0.2cm, align=center] (v11) at (6,8.25) {};
\node [ font={\huge\bfseries}, shape=circle, minimum size=0.5cm, circular drop shadow, text=black, very thick, draw=black!55, top color=white,bottom color=black!0, text width=0.2cm, align=center] (v12) at (6,6.75) {};
\node [ font={\huge\bfseries}, shape=circle, minimum size=0.5cm, circular drop shadow, text=black, very thick, draw=black!55, top color=white,bottom color=black!0, text width=0.2cm, align=center] (v13) at (6,5.25) {};
\node [ font={\huge\bfseries}, shape=circle, minimum size=0.5cm, circular drop shadow, text=black, very thick, draw=black!55, top color=white,bottom color=black!0, text width=0.2cm, align=center] (v14) at (6,3.75) {};
\node [ font={\huge\bfseries}, shape=circle, minimum size=0.5cm, circular drop shadow, text=black, very thick, draw=black!55, top color=white,bottom color=black!0, text width=0.2cm, align=center] (v15) at (6,2.25) {};
\node [ font={\huge\bfseries}, shape=circle, minimum size=0.5cm, circular drop shadow, text=black, very thick, draw=black!55, top color=white,bottom color=black!0, text width=0.2cm, align=center] (v16) at (6,0.75) {};
\node [ font={\huge\bfseries}, shape=circle, minimum size=0.5cm, circular drop shadow, text=black, very thick, draw=black!55, top color=white,bottom color=black!0, text width=0.2cm, align=center] (v17) at (6,-0.75) {};
\node [ font={\huge\bfseries}, shape=circle, minimum size=0.5cm, circular drop shadow, text=black, very thick, draw=black!55, top color=white,bottom color=black!0, text width=0.2cm, align=center] (v18) at (6,-2.25) {};
\node [ font={\huge\bfseries}, shape=circle, minimum size=0.5cm, circular drop shadow, text=black, very thick, draw=black!55, top color=white,bottom color=black!0, text width=0.2cm, align=center] (v19) at (6,-3.75) {};
\node [ font={\huge\bfseries}, shape=circle, minimum size=0.5cm, circular drop shadow, text=black, very thick, draw=black!55, top color=white,bottom color=black!0, text width=0.2cm, align=center] (v20) at (6,-5.25) {};
\node [ font={\huge\bfseries}, shape=circle, minimum size=0.5cm, circular drop shadow, text=black, very thick, draw=black!55, top color=white,bottom color=black!0, text width=0.2cm, align=center] (v21) at (6,-6.75) {};
\node [ font={\huge\bfseries}, shape=circle, minimum size=0.5cm, circular drop shadow, text=black, very thick, draw=black!55, top color=white,bottom color=black!0, text width=0.2cm, align=center] (v22) at (6,-8.25) {};
\draw (v11) edge (v5);
\draw (v5) edge (v12);
\draw (v6) edge (v13);
\draw (v6) edge (v14);
\draw (v10) edge (v15);
\draw (v10) edge (v16);
\draw (v7) edge (v17);
\draw (v7) edge (v18);
\draw (v8) edge (v19);
\draw (v8) edge (v20);
\draw (v9) edge (v21);
\draw (v9) edge (v22);
\end{tikzpicture}

%\end{comment}

\caption{A $(3,3,a)${-tree}   $ T_a$.}
\end{center}
\end{figure}

%\end{comment}

%\begin{defn}
%Let $k$ be a positive integer. A vertex set $S$ {\em has property $\mathcal{I}_{2k}$} if $S$ is $2k$-independent and each vertex $a \in S$ has a $(3,k,a)$-tree.
%\end{defn}

\begin{lemma}\label{C_2k}
Let $k$  be a  fixed positive integer and $x$ be a  real number with
$0<x < \frac{1}{3\cdot2^{k}-2}$. The number of pairings $F \in \mathcal{G}'_{2k+2}(n)$ such that $\pi(F)$
 has a $2k$-independent vertex set of size $xn$  is at most
$$q(n,k,x):={{n \choose xn} \cdot (3n-(6\cdot2^k-6)xn-1)!! } \cdot \prod\limits_{i=0}^{k-1}{(1-(3\cdot2^{i}-2)x)n \choose 3\cdot2^{i}xn } \cdot
(3\cdot2^{i}xn)! \cdot 3^{3\cdot2^{i}xn}.$$
\end{lemma}

\begin{proof}
%Recall that $|\mathcal{F}_3(n)|=(3n-1)!!$ and each vertex in an $F\in\mathcal{F}_3$ consists of $3$  points. 
To prove the lemma, we will show that  the total number of $2k$-independent  sets of size $xn$ in $\pi(F)$ over all $F \in \mathcal{G}'_{2k+2}(n)$ does not
 exceed  $q(n,k,x)$. Below we describe a procedure of constructing 
 for every  set  $C$ of size $xn$ in $[n]$ all pairings in $\mathcal{G}'_{2k+2}(n)$ for which $C$ is $2k$-independent.
  Not every obtained pairing will
be in  $ \mathcal{G}'_{2k+2}(n)$, but every $F \in \mathcal{G}'_{2k+2}(n)$ such that $C$
 is a $2k$-independent set in $\pi(F)$ 
 will be a result of this
 procedure:
 
\begin{enumerate}
\item
 We choose  a vertex set $C$ of size $xn$ from $[n]$. There are ${n \choose xn}$ ways to do it.
\item
In order $C$ to be $2k$-independent and $\pi(F)$ to have girth at  least $2k+2$, all the balls of radius $k$ with the centers in
$C$ must be disjoint, and for each $a\in C$,  the ball $B_{\pi(F)}(a,k)$ must induce a $(3,k,a)$-tree. Thus,
we have $(1-x)n \choose 3xn$ ways to choose the neighbors of $C$, call it $N(C)$, $(3xn)!$ ways to 
determine which vertex in $N(C)$ will be the neighbor for each point in $\pi^{-1}(C)$,
%order those $3xn$ neighbors of $C$
 and $3^{3xn}$ ways to decide which point of each vertex in $N(C)$ is adjacent to the corresponding point in $\pi^{-1}(C)$. 
 Each vertex of $N(C)$ will have $2$ free points left at this moment, and in total the set $\pi^{-1}(N(C))$ has now  $2 \cdot 3xn=6xn$ free points. 
\item
Similarly to the previous step, consecutively for  $i=1,2,\ldots,k-1$, we will decide which vertices and points are in the set $\pi^{-1}(N^{i+1}(C))$
of the vertices at distance $i$ from $C$, as follows. Before the $i$th iteration, we have $3x\cdot 2^{i}n$ free points in the
$3x\cdot 2^{i-1}n$ vertices of
$\pi^{-1}(N^{i}(C))$, and 
$$|C\cup N^1(C)\cup\ldots\cup N^{i}(C)|=xn\left(1+3(1+2+\ldots+2^{i-1})\right) = (3 \cdot 2^i - 2)xn.$$
We   choose $3x\cdot 2^{i}n$ vertices out of the remaining $\left(1-(3 \cdot 2^{i}-2)x\right)n$ 
vertices to include into $N^{i+1}(C)$, then we have $(3x\cdot 2^{i}n)!$ ways to 
determine which vertex in $N^{i+1}(C)$ will be the neighbor for each free point in $\pi^{-1}(N^{i}(C))$,
 and $3^{3x\cdot 2^{i}n}$ ways to decide which point of each vertex in $N^{i+1}(C)$ is adjacent to the corresponding point in 
 $\pi^{-1}(N^{i}(C))$. 
%, we have ($3 \cdot 2^{k-1}xn$)! ways to order those $3 \cdot 2^{k-1}xn$ vertices and $3^{3 \cdot 2^{k-1}xn}$ ways to pick which two  points of each vertex are left unpaired.
\item
 Finally, there are $3n-(6 \cdot 2^k - 6)xn$ free  points left and we have $(3n-(6 \cdot 2^k - 6)xn-1)!!$  ways to pair them. 
\end{enumerate}

This proves the bound.
\end{proof}

In the proofs below we will use Stirling's formula: { For every $n\geq 1$,}
\begin{equation}\label{SF}
\sqrt{2\pi n}\left(\frac{n}{e}\right)^n\leq n!\leq \sqrt{2\pi n}\left(\frac{n}{e}\right)^n\,e^{1/12n}.
\end{equation}
%In the calculations 
%we will care mostly about exponential factors and  neglect polynomial factors.

% Let $$\mathcal{I}_x=\{C|C\ is\ 2k-independent\ and\ each\ vertex\ a\in C\ has\ a\ rooted\ tree\ T_a\ of\ size\ 3 \cdot 2^k-2\}.$$

\begin{corl}\label{coreven}
Let $g\geq 22$ be fixed. For every $\epsilon > 0$, there exists an $N>0$ such that for each $n>N$, 
\begin{equation} \label{a}
{|\{  G \in \mathcal{G}_g(n)\, :\;  c_2(G)>0.236n:=b_2n\}|} \ <\epsilon\, \cdot {|\mathcal{G}_g(n)|},
\end{equation}

\begin{equation} \label{b}
{|\{G \in \mathcal{G}_g(n)\, :\;    c_4(G)>0.082n:=b_4n\}|} \ <\epsilon\, \cdot {|\mathcal{G}_g(n)|},
\end{equation}

\begin{equation} \label{c}
{|\{  G \in \mathcal{G}_g(n)\, :\;   c_6(G)>0.03n:=b_6n\}|} \ <\epsilon\, \cdot {|\mathcal{G}_g(n)|},
\end{equation}

\begin{equation} \label{d}
{|\{  G \in \mathcal{G}_g(n)\, :\;   c_8(G)>0.011n:=b_8n\}|} \ <\epsilon\, \cdot {|\mathcal{G}_g(n)|},
\end{equation}
and

\begin{equation} \label{e}
{|\{  G \in \mathcal{G}_g(n)\,:\;  c_{10}(G)>0.004n:=b_{10}n\}|} \ <\epsilon\, \cdot {|\mathcal{G}_g(n)|}.
\end{equation}
\end{corl}

\begin{proof}
By Lemma~\ref{C_2k},
$$q(n,k,x)={{n \choose xn} \cdot ((3n-(6\cdot2^k-6)xn-1)!!) }  \prod\limits_{i=0}^{k-1}{(1-(3\cdot2^{i}-2)x)n \choose 3\cdot2^{i}xn } \cdot
((3\cdot2^{i}xn)!) (3^{3\cdot2^{i}xn})$$

$$= \frac{(3n-(6\cdot2^k-6)xn-1)!!\cdot n!}{(xn)! \cdot ((1-x)n)!} \cdot 3^{3xn+6xn+\ldots +3 \cdot 2^{k-1}xn}$$

$$\cdot \frac{((1-x)n)! \cdot (3xn)!}{(3xn)! \cdot ((1-4x)n)!} \cdot \frac{((1-4x)n)! \cdot (6xn)!}{(6xn)! \cdot ((1-10x)n)!} \cdot \ldots  \cdot \frac{((1-(3 \cdot 2^{k-1} - 2)x)n)! \cdot (3 \cdot 2^{k-1}xn)!}{(3 \cdot 2^{k-1}xn)! \cdot ((1-(3 \cdot 2^k-2)x)n)!}$$

$$=\frac{(3n-(6\cdot2^k-6)xn-1)!!\cdot n!}{(xn)! \cdot ((1-(3 \cdot 2^k - 2)x)n)!} \cdot 3^{(3 \cdot 2^k - 3)xn}.$$

\noindent
We know that $$(3n-1)!! = \frac{(3n)!!}{3n}  \geq \frac{\sqrt{(3n)!}}{3n}$$
and
$$  (3n-(6 \cdot 2^k - 6)xn - 1)!! \leq \sqrt{(3n-(6 \cdot 2^k - 6)xn)!}.$$
Therefore,

$$\frac{q(n,k,x)}{(3n-1)!!} \leq (3n) \cdot   \left(\frac{(3n-(6 \cdot 2^k - 6)xn)!}{(3n)!} \right)^{\frac{1}{2}}   \cdot \frac{n!}{(xn)! \cdot ((1-(3 \cdot 2^k - 2)x)n)!} \cdot 3^{(3 \cdot 2^k - 3)xn}. $$
Using Stirling's formula~\eqref{SF},
we have

$$
\frac{q(n,k,x)}{(3n-1)!!}  =  O(n^2) \cdot \frac{\left(\frac{n}{e}\right)^{\frac{1}{2} \cdot (3n-(6 \cdot 2^k - 6)xn)} \cdot \left(\frac{n}{e}\right)^n }{\left(\frac{n}{e}\right)^{\frac{3n}{2}} \cdot \left(\frac{n}{e}\right)^{xn} \cdot \left(\frac{n}{e}\right)^{(1-(3\cdot2^k-2)x)n}} \cdot
\left(\frac{(1-(2^{k+1}-2)x)^{1.5-(3\cdot2^k - 3)x}}
 {x^x(1-(3\cdot2^k-2)x)^{1-(3\cdot2^k-2)x}}\right)^n$$ 

$$=O(n^2) \cdot  \left(\frac{(1-(2^{k+1}-2)x)^{1.5-(3\cdot2^k - 3)x}}
 {x^x(1-(3\cdot2^k-2)x)^{1-(3\cdot2^k-2)x}}\right)^n.$$

%as $n$ tends to infinity.\\
\noindent
Let 
\begin{equation}\label{0312}
f(x,k)=\frac{(1-(2^{k+1}-2)x)^{1.5-(3\cdot2^k - 3)x}}
 {x^x(1-(3\cdot2^k-2)x)^{1-(3\cdot2^k-2)x}},
 \end{equation}
 so that 
 \begin{equation}\label{03122}
\frac{q(n,k,x)}{|\mathcal{F}_3(n)|} = \frac{q(n,k,x)}{(3n-1)!!}  =  O(n^2) \, (f(x,k))^n.
 \end{equation}

 By plugging $x=0.236$ and $k=1$ into~\eqref{0312} (using a computer or a good calculator), we see that $0<f(0.236,1)<0.9964$. 
Since $f(x,1)$ is a smooth function for $0<x<1$, there exists $\delta_1$ such that $f(x,1)<0.9964$ for
all $x\in [0.236-\delta_1, 0.236]$. If $n>1/\delta_1$, then there exists an  $x_1=x_1(n)\in [0.236-\delta_1, 0.236]$ such that
$x_1n$ is an integer.
 By~\eqref{03122},
  $$
  \frac{q(n,1, x_1n)}{|\mathcal{F}_3(n)|}  =   O(n^2) \, (0.9964)^n\to 0\qquad \mbox{as}\;{n\to\infty}.
  $$
By the definition of $q(n,k,x)$,~\eqref{03112} and Corollary~\ref{MK2}, this implies~\eqref{a}.  
  
 Similarly, by plugging  the corresponding values of $x$ and $k$ into~\eqref{0312}, one can check that  
   $0<f(0.082,2)<0.9977$, $0<f(0.03,3)<0.9981$,  $0<f(0.011,4)<0.996$, and $0<f(0.004,5)<0.995$. 
 Thus repeating the argument of the previous paragraph, we obtain that  
    \eqref{b}, \eqref{c}, \eqref{d}, \eqref{e} also hold.
\end{proof}

\begin{lemma}\label{C_2k+1}
Let $k$ be a fixed positive integer and $0<x < \frac{0.45537}{2^{k+1}-1}$. The number of pairings $F \in \mathcal{G}'_{2k+2}(n)$ such that $\pi(F)$ has a $(2k+1)$-independent vertex set of size $xn$ is at most
$$
r(n,k,x):=\frac{{n \choose xn} \cdot (3(n-(3\cdot2^k-2)xn))! \cdot (3(n-(4\cdot2^k-2)xn)-1)!!
}{(3(n-(4\cdot2^k-2)xn))!}\qquad
$$
\begin{equation}\label{eq1}
\qquad\times
\prod_{i=0}^{k-1}{(1-(3\cdot2^{i}-2)x)n \choose 3\cdot2^{i}xn }  \cdot (3\cdot2^{i}xn)! \cdot 3^{3\cdot2^{i}xn}.
\end{equation}
\end{lemma}

\begin{proof}
 We will show that  the total number of $(2k+1)$-independent  sets of size $xn$ in $\pi(F)$ over all $F \in \mathcal{G}'_{2k+2}(n)$ does not
 exceed  $r(n,k,x)$.
%, which would imply the lemma.
 Below we describe a procedure of constructing 
 for every  set  $C$ of size $xn$ in $[n]$ all pairings in $\mathcal{G}'_{2k+2}(n)$ for which $C$ is $(2k+1)$-independent.
  Not every obtained pairing will
be in  $ \mathcal{G}'_{2k+2}(n)$, but every $F \in \mathcal{G}'_{2k+2}(n)$ such that $C$
 is a $(2k+1)$-independent set in $\pi(F)$ 
 will be a result of this
 procedure:

\begin{enumerate}
\item
We choose a vertex set $C$ of size $xn$ from $[n]$. There are $n \choose xn$ ways to do it.
\item

In order $C$ to be $(2k+1)$-independent and $\pi(F)$ to have girth at  least $2k+2$, all the balls of radius $k$ with the centers in
$C$ must be disjoint, and for each $a\in C$,  the ball $B_{\pi(F)}(a,k)$ must induce a $(3,k,a)$-tree. Thus,
we have $(1-x)n \choose 3xn$ ways to choose the neighbors of $C$, call it $N(C)$, $(3xn)!$ ways to 
determine which vertex in $N(C)$ will be the neighbor for each point in $\pi^{-1}(C)$,
%order those $3xn$ neighbors of $C$
 and $3^{3xn}$ ways to decide which point of each vertex in $N(C)$ is adjacent to the corresponding point in $\pi^{-1}(C)$. 
 Each vertex of $N(C)$ will have $2$ free points left at this moment, and in total the set $\pi^{-1}(N(C))$ has now  $2 \cdot 3xn=6xn$ free points. 

\item
Similarly to the previous step, consecutively for  $i=1,2,\ldots,k-1$, we will decide which vertices and points are in the set $\pi^{-1}(N^{i+1}(C))$
of the vertices at distance $i$ from $C$, as follows. Before the $i$th iteration, we have $3x\cdot 2^{i}n$ free points in the
$3x\cdot 2^{i-1}n$ vertices of
$\pi^{-1}(N^{i}(C))$, and 
$$|C\cup N^1(C)\cup\ldots\cup N^{i}(C)|=xn\left(1+3(1+2+\ldots+2^{i-1})\right) = (3 \cdot 2^i - 2)xn.$$
We   choose $3x\cdot 2^{i}n$ vertices out of the remaining $\left(1-(3 \cdot 2^{i}-2)x\right)n$ 
vertices to include into $N^{i+1}(C)$, then we have $(3x\cdot 2^{i}n)!$ ways to 
determine which vertex in $N^{i+1}(C)$ will be the neighbor for each free point in $\pi^{-1}(N^{i}(C))$,
 and $3^{3x\cdot 2^{i}n}$ ways to decide which point of each vertex in $N^{i+1}(C)$ is adjacent to the corresponding point in 
 $\pi^{-1}(N^{i}(C))$. 

\item
 Let $N^0(C) := C$ and $S := \cup_{i=0}^{k}N^i(C)$. In order the distance between each pair of vertices in $C$ to be at least $2k+2$, $N^k(C)$ has to be an independent set. Therefore, each of the $3x \cdot 2^k n$  free points in the $3x \cdot 2^{k-1} n$ vertices of $\pi^{-1}(N^k(C))$ has to be paired with one of the remaining $ 3(n-(3\cdot2^k-2)xn)$ free points of $\pi^{-1}([n] - S)$ 
and we have $$\frac{(3(n-(3\cdot2^k-2)xn))!}{(3(n-(4\cdot2^k-2)xn))!}$$  ways to do that. 
\item
Finally, there are $3(n-(4\cdot2^k-2)xn)$  free points left and we have $(3(n-(4\cdot2^k-2)xn)-1)!!$  ways to pair them.

\end{enumerate}

\end{proof}
 
%\newpage

%Some claims below will use Stirling's formula \eqref{SF}. In the calculation we will care mostly about exponential factors and neglect %polynomial factors.

\begin{corl}\label{corodd}
Let $g\geq 24$ be fixed. For every $\epsilon > 0$, there exists an $N>0$ such that for each $n>N$, 
\begin{equation} \label{p}
{|\{  G \in \mathcal{G}_g(n)\, :\;  c_3(G)>0.1394n:=b_3n\}|} \ <\epsilon\, \cdot {|\mathcal{G}_g(n)|},
\end{equation}

\begin{equation} \label{q}
{|\{G \in \mathcal{G}_g(n)\, :\;    c_5(G)>0.05n:=b_5n\}|} \ <\epsilon\, \cdot {|\mathcal{G}_g(n)|},
\end{equation}

\begin{equation} \label{r}
{|\{  G \in \mathcal{G}_g(n)\, :\;   c_7(G)>0.0182n:=b_7n\}|} \ <\epsilon\, \cdot {|\mathcal{G}_g(n)|},
\end{equation}

\begin{equation} \label{s}
{|\{  G \in \mathcal{G}_g(n)\, :\;   c_9(G)>0.0063n:=b_9n\}|} \ <\epsilon\, \cdot {|\mathcal{G}_g(n)|},
\end{equation}
and

\begin{equation} \label{t}
{|\{  G \in \mathcal{G}_g(n)\,:\;  c_{11}(G)>0.0022n:=b_{11}n\}|} \ <\epsilon\, \cdot {|\mathcal{G}_g(n)|}.
\end{equation}
\end{corl}

\begin{proof} 
By Lemma \ref{C_2k+1},

$$
r(n,k,x)=\frac{{n \choose xn} \cdot (3(n-(3\cdot2^k-2)xn))! \cdot (3(n-(4\cdot2^k-2)xn)-1)!!
}{  (3(n-(4\cdot2^k-2)xn))!}\qquad
$$
\begin{equation}
\qquad\times
\prod_{i=0}^{k-1}{(1-(3\cdot2^{i}-2)x)n \choose 3\cdot2^{i}xn }  \cdot (3\cdot2^{i}xn)! \cdot 3^{3\cdot2^{i}xn}
\end{equation}

$$=\frac{ (3(n-(3\cdot2^k-2)xn))! \cdot (3(n-(4\cdot2^k-2)xn)-1)!!
}{(3(n-(4\cdot2^k-2)xn))!} \cdot \frac{n!}{(xn)! \cdot ((1-x)n)!} \cdot 3^{3xn+6xn+\ldots +3 \cdot 2^{k-1}xn}$$

$$ \cdot \frac{((1-x)n)! \cdot (3xn)!}{(3xn)! \cdot ((1-4x)n)!} \cdot \frac{((1-4x)n)! \cdot (6xn)!}{(6xn)! \cdot ((1-10x)n)!} \cdot \ldots  \cdot \frac{((1-(3 \cdot 2^{k-1} - 2)x)n)! \cdot (3 \cdot 2^{k-1}xn)!}{(3 \cdot 2^{k-1}xn)! \cdot ((1-(3 \cdot 2^k-2)x)n)!}$$

$$=\frac{ (3(n-(3\cdot2^k-2)xn))! \cdot (3(n-(4\cdot2^k-2)xn)-1)!!
}{(3(n-(4\cdot2^k-2)xn))!} \cdot \frac{n!}{(xn)! \cdot ((1-(3 \cdot 2^k - 2)x)n)!} \cdot 3^{(3 \cdot 2^k - 3)xn}$$

We know that $$(3n-1)!! \geq \frac{(3n)!!}{3n}  \geq \frac{\sqrt{(3n)!}}{3n}$$

and

$$   (3(n-(4\cdot2^k-2)xn)-1)!! \leq \sqrt{ (3(n-(4\cdot2^k-2)xn))!}.$$

Therefore,

$$\frac{r(n,k,x)}{(3n-1)!!} \leq (3n) \cdot   \left(\frac{ (3(n-(4\cdot2^k-2)xn))!}{(3n)!} \right)^{\frac{1}{2}} \cdot \frac{(3(n-(3\cdot2^k-2)xn))!}{(3(n-(4\cdot2^k-2)xn))!} $$

$$ \cdot \frac{n!}{(xn)! \cdot ((1-(3 \cdot 2^k - 2)x)n)!} \cdot 3^{(3 \cdot 2^k - 3)xn}. $$

By Stirling's formula $\eqref{SF}$,

 $$\frac{r(n,k,x)}{(3n-1)!!} = O(n^3) \cdot \frac{\left(\frac{n}{e}\right)^{\frac{3}{2} \cdot (n-(4 \cdot 2^k - 2)xn)} \cdot \left(\frac{n}{e}\right)^{3(n-(3 \cdot 2^k - 2)xn)} \cdot  \left(\frac{n}{e}\right)^{n} }{ \left(\frac{n}{e}\right)^{\frac{3n}{2}} \cdot \left(\frac{n}{e}\right)^{3(n-(4 \cdot 2^k - 2)xn)} \cdot \left(\frac{n}{e}\right)^{xn} \cdot \left(\frac{n}{e}\right)^{(1-(3 \cdot 2^k - 2)x)n}} $$
$$ \cdot \left(\frac{(1-(3\cdot2^k-2)x)^{2-(6\cdot2^k-4)x}}{x^x(1-(4\cdot2^k-2)x)^{1.5-(6\cdot2^k-3)x}}\right)^n$$ 
$$=O(n^3) \cdot \left(\frac{(1-(3\cdot2^k-2)x)^{2-(6\cdot2^k-4)x}}{x^x(1-(4\cdot2^k-2)x)^{1.5-(6\cdot2^k-3)x}}\right)^n.$$
Let
\begin{equation}\label{hxk}
h(x,k)=\frac{(1-(3\cdot2^k-2)x)^{2-(6\cdot2^k-4)x}}{x^x(1-(4\cdot2^k-2)x)^{1.5-(6\cdot2^k-3)x}},
\end{equation}
so that

\begin{equation}\label{rf3}
\frac{r(n,k,x)}{|\mathcal{F}_3(n)|} = \frac{r(n,k,x)}{(3n-1)!!} = O(n^3)(h(x,k))^n.
\end{equation}

By plugging $x=0.1394$ and $k=1$ into $\eqref{hxk}$ (using a computer or a good calculator), we see that $0<h(0.1394,1)<0.9974$. 
Since $h(x,1)$ is a smooth function for $0<x<1$, there exists $\nu_1$ such that $h(x,1)<0.9974$ for
all $x\in [0.1394-\nu_1, 0.1394]$. If $n>1/\nu_1$, then there exists an  $x_1=x_1(n)\in [0.1394-\nu_1, 0.1394]$ such that
$x_1n$ is an integer.
By $\eqref{rf3}$,
$$ \frac{r(n,1,x_1n)}{|\mathcal{F}_3(n)|} = O(n^3)(0.9974)^n \to 0\qquad \mbox{as}\;{n\to\infty}.$$ 
By the definition of $r(n,k,x)$, \eqref{03112} and Corollary \ref{MK2}, this implies \eqref{p}.

Similarly, by plugging the corresponding values of $x$ and $k$ into $\eqref{hxk}$, one can check that 
$0<h(0.05,2)<0.9985$, $0<h(0.0182,3)<0.9973$, $0<h(0.0063,4)<0.9986$, and $0<h(0.0022,5)<0.9979$. Thus repeating the argument of the previous paragraph, we obtain that \eqref{q}, \eqref{r}, \eqref{s}, \eqref{t} also hold.

\end{proof}

\section{Bound on $|C_1 \cup C_2 \cup C_4|$}

\begin{defn}
For a fixed graph $G$, let  $c_{1,2,4}(G)$ be the maximum size  of $|C_1 \cup C_2 \cup C_4|$, where
$C_1$, $C_2$ and $C_4$ are disjoint  subsets of $V(G)$ such that $C_i$ is $i$-independent for all $i\in\{1,2,4\}$.
%and $C_4$ is $4$-independent.
\end{defn}

In this section we prove an upper bound for $c_{1,2,4}(G)$. %, which is smaller than the bound we have for $c_1+c_2+c_4$. 
%We are always assuming that the girth $g$ is large and speaking about the bound in the sense of asymptotically almost surely true in this section.
% By Theorem 7, we have 
%\begin{equation}\label{c1}
%c_1<0.456n.
%\end{equation}

\begin{lemma}\label{bigl}
Let $G$ be an $n$-vertex cubic graph with girth at least $9$ and 
\begin{equation}\label{c1}
c_1(G)<0.456n.
\end{equation}
Then $c_{1,2,4}(G)\leq 0.7174n:=b_{1,2,4}n.$
\end{lemma}

{\bf Proof.} Let $G$ satisfy the conditions of the lemma, and let
$C_1$, $C_2$ and $C_4$ be disjoint  subsets of $V(G)$ such that $C_i$ is $i$-independent for $i\in \{1,2,4\}$ and
$|C_1 \cup C_2 \cup C_4|=c_{1,2,4}(G)$. 

Since  $C_2$ is $2$-independent, each vertex in $C_1$ has at most one neighbor in $C_2$.
Let $Q$ be the set of vertices 
  in $C_1$ 
that do not have neighbors in $C_2$, and $q=|Q|$.
%Note that a vertex in $C_1$ has at most one neighbor in $C_2$ because $C_2$ is $2$-independent. 
Let $L$ be the set of edges in $G-C_1-C_2$ and
$\ell=|L|$. For brevity, the vertices in $Q$ will be called $Q$-{\em vertices}, and the 
 edges in $L$ will be called
$L$-{\em edges}. Let $s=|C_1|+|C_2|$.

We will prove the lemma in a series of claims. Our first claim is:

\begin{equation}\label{s4}
s < 0.652 n.
\end{equation}

To show~\eqref{s4}, we  count the edges connecting $C_1\cup C_2$ with $\overline{C_1\cup C_2}$ in two ways: 
\begin{equation}\label{edge}
3(n -s) - 2\ell = e[C_1\cup C_2, \overline{C_1\cup C_2}] = 3s - 2(|C_1|-q).
\end{equation}
Solving for $s$, we get $s=\frac{n}{2}-\frac{1}{3}(\ell-|C_1|+q)$. Since $q,\ell\geq 0$ and $|C_1|\leq c_1$, this
together with~\eqref{c1}
yields
$$s\leq \frac{n}{2}-\frac{1}{3}(0-|C_1|+0)\leq \frac{n}{2}+\frac{c_1}{3}<0.652 n,$$
as claimed.
\qed
%\end{proof}

\medskip

%We can push the argument further if we involve $C_4$. We know that each vertex $a$ in $C_4$ has a $(3,2,a)$-tree $T_a$ with $10$ vertices, 
%that are vertex disjoint for every $a\in C_4$. \\

%picture

\begin{figure}[ht]
\begin{center}

%\begin{comment}

      \tikzstyle{green_circle} = [  font={\huge\bfseries}, shape=circle, minimum size=0.5cm, circular drop shadow, text=black, very thick, draw=black!55, top color=white,bottom color=green!80, text width=0.5cm, align=center]
    \usetikzlibrary{shadows}
\begin{tikzpicture}[rotate=270]

\node [ font={\huge\bfseries}, shape=circle, minimum size=0.4cm, circular drop shadow, text=black, very thick, draw=black!55, top color=white,bottom color=green!0, text width=0.8cm, align=center] (v1) at (0,0) {$a$};
\node [ font={\huge\bfseries}, shape=circle, minimum size=0.4cm, circular drop shadow, text=black, very thick, draw=black!55, top color=white,bottom color=green!0, text width=0.8cm, align=center] (v2) at (2,4) {$a_{\text{\small{3}}}$};
\node [ font={\huge\bfseries}, shape=circle, minimum size=0.4cm, circular drop shadow, text=black, very thick, draw=black!55, top color=white,bottom color=green!0, text width=0.8cm, align=center] (v3) at (2,0) {$a_{\text{\small{2}}}$};
\node [ font={\huge\bfseries}, shape=circle, minimum size=0.4cm, circular drop shadow, text=black, very thick, draw=black!55, top color=white,bottom color=green!0, text width=0.8cm, align=center] (v4) at (2,-4) {$a_{\text{\small{1}}}$};
\node [ font={\huge\bfseries}, shape=circle, minimum size=0.8cm, circular drop shadow, text=black, very thick, draw=black!55, top color=white,bottom color=green!0, text width=0.8cm, align=center] (v5) at (4,5) {$a_{\text{\footnotesize{3,2}}}$};
\node [ font={\huge\bfseries}, shape=circle, minimum size=0.8cm, circular drop shadow, text=black, very thick, draw=black!55, top color=white,bottom color=green!0, text width=0.8cm, align=center] (v6) at (4,3) {$a_{\text{\footnotesize{3,1}}}$};
\node [ font={\huge\bfseries}, shape=circle, minimum size=0.8cm, circular drop shadow, text=black, very thick, draw=black!55, top color=white,bottom color=green!0, text width=0.8cm, align=center] (v7) at (4,1) {$a_{\text{\footnotesize{2,2}}}$};
\node [ font={\huge\bfseries}, shape=circle, minimum size=0.8cm, circular drop shadow, text=black, very thick, draw=black!55, top color=white,bottom color=green!0, text width=0.8cm, align=center] (v8) at (4,-1) {$a_{\text{\footnotesize{2,1}}}$};
\node [ font={\huge\bfseries}, shape=circle, minimum size=0.8cm, circular drop shadow, text=black, very thick, draw=black!55, top color=white,bottom color=green!0, text width=0.8cm, align=center] (v9) at (4,-3) {$a_{\text{\footnotesize{1,2}}}$};
\node [ font={\huge\bfseries}, shape=circle, minimum size=0.8cm, circular drop shadow, text=black, very thick, draw=black!55, top color=white,bottom color=green!0, text width=0.8cm, align=center] (v10) at (4,-5) {$a_{\text{\footnotesize{1,1}}}$};
\draw  (v1) edge (v2);
\draw  (v1) edge (v3);
\draw  (v1) edge (v4);
\draw  (v2) edge (v5);
\draw  (v2) edge (v6);
\draw  (v3) edge (v7);
\draw  (v3) edge (v8);
\draw  (v4) edge (v9);
\draw  (v10) edge (v4);
\end{tikzpicture}

%\end{comment}

\caption{A $(3,2,a)${-tree}   $ T_a$.}\label{fi2}
\end{center}
\end{figure}
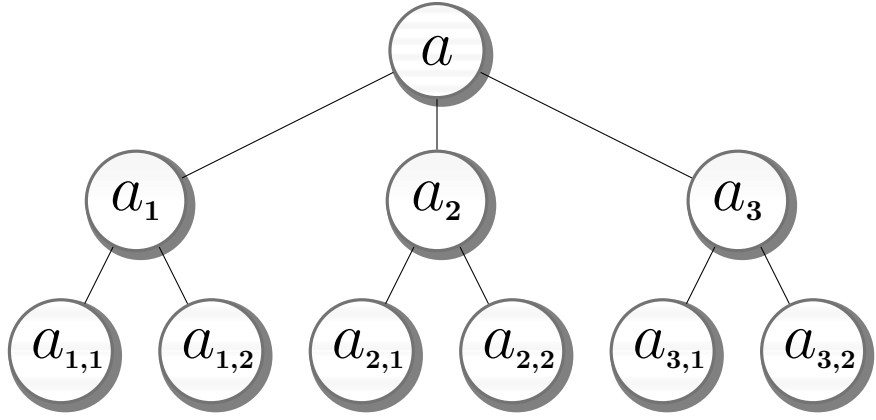

Since $g(G)\geq 9$, for every $a\in V(G)$, the ball $B_G(a,2)$ induces a $(3,2,a)$-tree $T_a$.
When handling such a tree $T_a$, we will use the following notation (see Fig~\ref{fi2}):
 $$V(T_a)=\{a\}\cup N_1(a)\cup N_2(a),\, \mbox{ where}\, N_1(a)=\{a_1,a_2,a_3\},\,N_2(a)= \{ a_{1,1}, 
a_{1,2},a_{2,1},a_{2,2},a_{3,1},a_{3,2}\},$$ and 
$$E(T)=\{aa_1,aa_2,aa_3,a_1a_{1,1},a_1a_{1,2},a_2a_{2,1},a_2a_{2,2}, a_3a_{3,1},a_3a_{3,2}\}.$$

%$$\text{Figure 2: A }(3,2,a)\text{-tree}\ T_a\ \text{with\ 10\ vertices.}$$
%\end{comment}

%For a vertex $a\in C_4$, we will use $T_a$ to denote the $(3,2,a)$-tree of $a$ and we call $T_a$ the tree of $a$. 
\noindent
For $j\in \{0,1,2\}$, let
$$S_j=\{a\in C_4:\, \mbox{ the total number  of  $L$-edges and  $Q$-vertices in $T_a$ is } j\},$$
and let $U=C_4-\bigcup_{j=0}^2 S_j$.

\medskip
Our next claim is:
\begin{equation}\label{3edges}
\parbox{5.5in}{\em 
 For each $0\leq j\leq 2$ and every $a\in S_j$, 
$|V(T_a)\cap C_2|\geq 3-j$.% contains at least $3$ vertices from $C_2$.
}
\end{equation}

%\begin{proof}
Indeed, let  $0\leq j\leq 2$ and $a \in S_j$. 
If a vertex $a_i\in N_1(a)$ is not in $ (C_1 \cup C_2)-Q$, then either $a_i\in Q$ or $aa_i\in L$.
Thus, by the definition of $S_j$, $|N_1(a)\cap ((C_1 \cup C_2)-Q)|\geq 3-j$.
 Since  each  $a_i\in (C_1 \cup C_2)-Q$ either is in $C_2$ or has a neighbor in $C_2\cap \{a_{i,1},a_{i,2}\}$,
 we get at least $3-j$ vertices  in $C_2\cap V(T_a)$. This proves~\eqref{3edges}.

For $0\leq j\leq 2$,  let $|S_j|=\alpha_j n$, and let  $|U|=\beta n$. Then

\begin{equation}\label{c4}
(\alpha_1+\alpha_2+\alpha_3+\beta)n=|C_4|.
\end{equation}

By the definition of $4$-independent sets,  for all $a\in C_4$  the balls $B_G(a,2)$ are disjoint  and not adjacent to each other.
For $0\leq j\leq 2$ and every  $a\in S_j$, the tree $T_a$
  contributes $j$ to $\ell+q$, and for every $a\in U$,  $T_a$ contributes at least $3$ to $\ell+q$. Therefore
\begin{equation}\label{ellt}
\alpha_1n+2\alpha_2 n+3\beta n\le \ell+q.
\end{equation}
Also,~\eqref{3edges}  yields a lower bound on $|C_2|$:
\begin{equation}\label{lbc2}
3\alpha_0 n+2 \alpha_1n+\alpha_2n\le |C_2| .
\end{equation}

\noindent
Now~\eqref{c4}, \eqref{ellt}, and  \eqref{lbc2} yield
\begin{equation}\label{0318}
3|C_4|=(\alpha_1n+2\alpha_2 n+3\beta n)+(3\alpha_0 n+2 \alpha_1n+\alpha_2n)\leq \ell+q+|C_2|.
\end{equation}
On the other hand,  by~\eqref{edge}
$$2(\ell+q)=3n-6s+2|C_1|=3n-4s-2|C_2|,$$
 so
$2(\ell+q+|C_2|)=3n-4s$. Comparing with~\eqref{0318}, we get
$$|C_4|\leq \frac{3n-4s}{6}=\frac{3n+2s}{6}-s.$$
Hence by the definition of $s$ and~\eqref{s4},
$$|C_1 \cup C_2 \cup C_4|=|C_4|+s\leq \frac{3n+2s}{6}\leq \frac{n}{2}+\frac{0.652n}{3}\leq 0.7174n.\qed
$$

\section{Proof of Theorem 1}
%Recall that $\mathcal{G}_g(n)$ is the set of all $n$-vertex cubic graphs with girth at least $g$. 
Let $J:=\{3,5,6,7,\ldots,11\}$ and 
\begin{equation}\label{Bg}
\mathcal{B}_g(n)     =     \left\{ G \in \mathcal{G}_g(n) \, :\; c_{1,2,4}(G)+ \sum_{j\in J}c_j(G) >0.9785n\;\mbox{ or }\;
\sum_{j=6}^{\lceil k/2\rceil - 1  }c_{2j+1}>\frac{2 \cdot 0.45537n}{127}.
%c_1(G)>0.45537n
\right\}.
\end{equation}

\begin{lemma}\label{conclusion}
 Let $k \geq 12$ be a fixed  integer and $g \geq 2k+2$.  For every $\epsilon > 0$, there exists an $N=N(\epsilon) > 0$ such that for each $n > N$, 
\begin{equation}\label{1,2,4,3,5}
|\mathcal{B}_g(n)| <\epsilon \cdot |\mathcal{G}_g(n)|.
\end{equation}
\end{lemma}

\begin{proof}
% Without loss of generality, let $k \geq 12$ be even, we have a similar proof for odd $k$. 
Let $\epsilon > 0$ be given.
 By Lemma~\ref{bigl}, Theorem~\ref{MK1}, and Corollary~\ref{MK2}, there exists an $N_{1,2,4} > 0$ such that for each $n>N_{1,2,4}$,         
$$|\{ G \in \mathcal{G}_g(n)\, :\;c_{1,2,4}(G)>0.7174n\}| \ < \frac{\epsilon}{10}  \cdot |\mathcal{G}_g(n)|.$$
 Let $$M_{1,2,4}(n):=\{  G \in \mathcal{G}_g(n)\, :\;c_{1,2,4}(G)>0.7174n\}.$$ 
For each $j\in J$ and the constants
$b_j$ defined in Corollaries~\ref{coreven} and~\ref{corodd}, let 
$$M_{j}(n):=\{  G \in \mathcal{G}_g(n)\, :\;c_{j}(G)>b_jn\}.$$  
Let 
$$\mathcal{B}'_g(n)     =     \left\{ G \in \mathcal{G}_g(n)  :\; c_{1,2,4}(G)+ \sum_{j\in J}c_j(G) >0.9785n%c_1(G)>0.45537n
\right\}$$
$$
\mbox{and }\quad
\mathcal{B}''_g(n)     =     \left\{ G \in \mathcal{G}_g(n) :
c_1(G)>0.45537n
\right\}.$$
Since $b_{1,2,4}n+\sum_{j\in J} b_{j}n=0.9785n$ and 
%We have for those $G \in \mathcal{G}_g(n)$ with 
$c_{1,2,4}+ \sum_{j\in J}c_j>0.9785n,$ 
if $G \in \mathcal{B}'_g(n)$, then
$$G \in M_{1,2,4}(n) \cup\bigcup_{j\in J} M_j(n).$$ 
 Corollaries~\ref{coreven} and~\ref{corodd} imply that for each $j\in J$, there exists an $N_j > 0$ such that for each $n > N_j$, 
$$|\{ G \in \mathcal{G}_g(n) \, :\; c_j(G)>b_jn\}| \ < \frac{\epsilon}{10} \cdot |\mathcal{G}_g(n)|.$$
By Theorem~\ref{MK1}, there exists an $N_1> 0$ such that for each $n > N_1$, 
$| \mathcal{B}''_g(n)| \ < \frac{\epsilon}{10} \cdot |\mathcal{G}_g(n)|.$

Let $N=\max\{N_{1,2,4}, N_1,N_3,N_5,N_6,\ldots ,N_{11}\}$. 
%and $$\mathcal{B}_g(n)     =     \{ G \in \mathcal{G}_g(n) \, :\; c_{1,2,4}(G)+ \sum_{j\in J}c_j(G) >0.9785n\;\mbox{ or }\; c_1(G)>0.45537n\}.$$
By the definition of $N$,  for each $n > N$,
\begin{equation}\label{0325}
 |\mathcal{B}'_g(n)| + |\mathcal{B}''_g(n)| < (1+|J|+1)\frac{\epsilon}{10} \cdot |\mathcal{G}_g(n)|=\epsilon \cdot |\mathcal{G}_g(n)|.
 \end{equation}
%  implying $\eqref{1,2,4,3,5}$. 

%Applying Lemma \ref{trivial bounds} to the remaining even terms, we get $$\sum_{j=6}^{\lceil k/2\rceil } c_{2j}
 %\leq \sum\limits_{k=6}^{\infty} \frac{n}{3 \cdot 2^k-2} < \frac{n}{190} \cdot \sum\limits_{k=0}^{\infty} \frac{1}{2^k}=\frac{n}{95},$$
%implying $\eqref{12,14}$.

\noindent
Every graph $G\in \mathcal{G}_g(n)\setminus \mathcal{B}''_g(n)$ satisfies $c_1(G)\leq 0.45537n$.
Hence by Lemma \ref{trivial bounds}(ii), such a graph satisfies
$$\sum_{j=6}^{\lceil k/2\rceil - 1 }c_{2j+1} < \sum\limits_{k=6}^{\infty} \frac{0.45537n}{2^{k+1}-1} \leq  \frac{0.45537n}{127} \cdot \sum\limits_{k=0}^{\infty} \frac{1}{2^k}=\frac{2 \cdot 0.45537n}{127}.$$
It follows that $\mathcal{B}_g(n)\subseteq \mathcal{B}'_g(n)\cup \mathcal{B}''_g(n)$. Thus~\eqref{0325} implies~\eqref{1,2,4,3,5}.
\end{proof}

Now we are prepared to prove our main result.
\begin{proof}[{\bf Proof of Theorem 1.}] Let $k\geq 12$ be a   fixed  integer  and $g \geq 2k+2$.
%Theorem~1 is equivalent to the statement that for each fixed  integer $k\geq 12$ and $g \geq 2k+2$, 
 We need to show that for every $\epsilon > 0$, there exists an $N > 0$ such that for each $n > N$,
\begin{equation}\label{03251}
|\{ G \in\ \mathcal{G}_g(n) \, :\; \chi_p(G)\ \leq\ k\}| < \epsilon \cdot |\mathcal{G}_g(n)|.
 \end{equation}

 Let $\epsilon > 0$ be given and $G\in\ \mathcal{G}_g(n)$ satisfy $\chi_p(G) \leq k$. Then there is a partition of $V(G)$ into $C_1,C_2,\ldots ,C_k$ such that for each $i=1,2,\ldots ,k$, $C_i$ is $i$-independent. In particular,   $|C_1|+|C_2|+\ldots +|C_k|=n.$
By Lemma \ref{trivial bounds}(i),
\begin{equation}\label{12,14}
\sum_{j=6}^{\lfloor k/2\rfloor } |C_{2j}|
 < \sum\limits_{k=6}^{\infty} \frac{n}{3 \cdot 2^k-2} < \frac{n}{190} \cdot \sum\limits_{k=0}^{\infty} \frac{1}{2^k}=\frac{n}{95}.
\end{equation}
Since $n-\frac{n}{95}>0.9785n+\frac{2 \cdot 0.45537n}{127}$, this implies that $G\in\mathcal{B}_g(n)$, where $\mathcal{B}_g(n)$
is defined by~\eqref{Bg}. Thus, Lemma~\ref{conclusion} implies~\eqref{03251}.
\end{proof}

\bigskip
\noindent
{\bf Remark.} It seems that with a bit more sophisticated calculations, one can prove the claim of Theorem~\ref{e1} not only for almost all
cubic graphs with girth at least $2k+2$, but for almost all cubic $n$-vertex graphs.

%\newpage
			
\end{document}